\documentclass[10pt]{article}
\font\smallit=cmti10
\font\smalltt=cmtt10

\usepackage{amssymb,latexsym,amsmath,epsfig,amsthm,cleveref,enumitem} 

\makeatletter

\renewcommand\section{\@startsection {section}{1}{\z@}
{-30pt \@plus -1ex \@minus -.2ex}
{2.3ex \@plus.2ex}
{\normalfont\normalsize\bfseries\boldmath}}

\renewcommand\subsection{\@startsection{subsection}{2}{\z@}
{-3.25ex\@plus -1ex \@minus -.2ex}
{1.5ex \@plus .2ex}
{\normalfont\normalsize\bfseries\boldmath}}

\renewcommand{\@seccntformat}[1]{\csname the#1\endcsname. }

\makeatother

\newtheorem{theorem}{Theorem}
\newtheorem{lemma}{Lemma}
\newtheorem{proposition}{Proposition}
\newtheorem{corollary}{Corollary}

\theoremstyle{definition}

\newtheorem{conjecture}{Conjecture}

\newtheorem{question}{Question}

\newcommand\cl[1]{\mathcal{#1}}
\newcommand{\Z}{\mathbb{Z}}
\newcommand{\eq}{\;=\;}
\newcommand{\deq}{\;:=\;}
\def\thin{{\hskip 1pt}}

\begin{document}

\begin{center}
\uppercase{\bf \boldmath The $2$-complexity of even positive integers}
\vskip 20pt
{\bf Pengcheng Zhang 
}\\
{\smallit Max Planck Institute for Mathematics, Vivatsgasse 7, Bonn, Germany}\\
{\tt pzhang@mpim-bonn.mpg.de}
\end{center}
\vskip 20pt
\centerline{\smallit Received: , Revised: , Accepted: , Published: } 
\vskip 30pt

\centerline{\bf Abstract}
\noindent
The question of integer complexity asks about the minimal number of $1$'s that are needed to express a positive integer using only addition and multiplication (and parentheses). In this paper, we propose the notion of $l$-complexity of multiples of~$l$, which specializes to integer complexity when $l=1$, prove several elementary results on $2$-complexity of even positive integers, and raise some interesting questions on $2$-complexity and in general $l$-complexity.

\pagestyle{myheadings}
\markright{\smalltt INTEGERS: 24 (2024)\hfill}
\thispagestyle{empty}
\baselineskip=12.875pt
\vskip 30pt

\section{Introduction}

Given a positive integer $n$, the \textit{integer complexity} of $n$, denoted as $\|n\|$, is defined as the minimal number of $1$'s that are needed to express $n$ in terms of $1$ using only addition and multiplication (and parentheses). Alternatively, one could also define the integer complexity recursively via
\begin{align*}
    \|1\|:=1\quad\text{ and }\quad\|n\|\deq\min_{\substack{a,b\in\Z^+  \\ a+b=n\text{ or }ab=n }}\big(\|a\|+\|b\|\big).
\end{align*}
The notion of integer complexity was first raised by Mahler--Popken \cite{mahler-popken}, and it has been a longstanding problem to determine the integer complexity of some given $n$. Integer complexity grows logarithmically:
\begin{align*}
    3\log_3n\;\leq\;\|n\|\;\leq\;3\log_2n.
\end{align*}
According to Guy \cite{guy}, Selfridge proved that the lower bound can be attained when $n=3^m$ and raised the question of whether there exists $a\in\Z^+$ such that $\|2^a\|<2a$. This question is nowadays usually incorporated into the following conjecture.

\begin{conjecture}
\label{l=1-conj}
For $a\geq 1$ and $b\geq 0$, $\|2^a\cdot 3^b\|=2a+3b$.
\end{conjecture}

There has been much progress on this topic, specifically on improving the upper bound for `generic' positive integers, algorithms for computing integer complexity, and partial results on \Cref{l=1-conj}. Here we refer to the papers by Steinerberger \cite{steinerberger}, Cordwell et al. \cite{algorithm-integer-complexity}, Altman--Zelinsky \cite{altman-zelinsky}, and Altman \cite{altman-2019} for discussions of previous results and recent advances on various aspects. We would also like to mention the paper by Arias de Reyna \cite{arias-de-reyna}, which raised many other conjectures on integer complexity, though most of them have been settled.

Now, let $l$ be a positive integer. Given a positive integer $n\in l\Z^+$, define the \textit{$l$-complexity} of~$n$, denoted as $\|n\|_l$, as the minimal number of $l$'s that are needed to express $n$ in terms of $l$ using only addition and multiplication (and parentheses). As before, one could also define the $l$-complexity recursively via
\begin{align*}
    \|l\|_l:=1\quad\text{ and }\quad\|n\|_l\deq\min_{\substack{a,b\in l\Z^+  \\ a+b=n\text{ or }ab=n}}\big(\|a\|_l+\|b\|_l\big).
\end{align*}
Indeed, $l$-complexity specializes to integer complexity when $l=1$, and in this article, we will mostly focus on $l=2$. Our main result is a complete classification of $n\in 2\Z^+$ with $\|n\|_2=m+1$ and $\|n\|_2=m+2$, where $m=\lceil\log_2n\rceil-1$.

\begin{theorem}
\label{2-classify}
Let $n\in 2\Z^+$ and $m=\lceil\log_2n\rceil-1$, i.e., $2^m<n\leq 2^{m+1}$.
\begin{enumerate}[label=\rm{(\arabic*)}]
    \item $\|n\|_2=m+1$ if and only if $n=2^{m+1}$ or $n=2^m+2^{m^\prime}$ for $1\leq m^\prime<m$.
    \item $\|n\|_2=m+2$ if and only if $n$ is of one of the following forms
    \begin{enumerate}[label=\rm{(\alph*)}]
        \item $2^{m_1}+2^{m_2}+2^{m_3}$ for $m=m_1>m_2>m_3\geq 1$;
        \item $2^{m_1}+2^{m_2}+2^{m_3}+2^{m_4}$ for $m=m_1>m_2>m_3>m_4\geq 2$ with $m_1+m_4=m_2+m_3$;
        \item $2^{m_1}+2^{m_1-3}+2^{m_2}+2^{m_2-1}$ for $m=m_1\geq m_2+3\geq 6$;
        \item $2^m+2^{m-5}+2^{m-6}+2^{m-7}$ for $m\geq 10$.
    \end{enumerate}
\end{enumerate}
\end{theorem}

Using this classification, we are able to prove the following result on the $2$-complexity of even positive integers of certain particular forms.

\begin{theorem}
\label{6r-10r}
For $m\geq0$, we have
\begin{enumerate}[label=\rm{(\arabic*)}]
    \item $\|2^m\cdot 6^r\|_2=m+3r$ for $1\leq r\leq 7$, and $m+3r-1\leq\|2^m\cdot 6^r\|_2\leq m+3r$ for $8\leq r\leq 9$;
    \item $\|2^m\cdot 10^r\|_2=m+4r \text{ for }1\leq r\leq 4$, and $m+19\leq\|2^m\cdot 10^5\|_2\leq m+20$.
\end{enumerate}
\end{theorem}

We would also like to propose the following conjecture cautiously, which may well be entirely wrong for large $r$. One should compare this conjecture with \Cref{l=1-conj} in the realm of integer complexity. In \Cref{questions}, we also propose a list of other interesting questions to investigate.

\begin{conjecture}
\label{l=2-conj}
For $m\geq0$ and $r\geq1$, $\|2^m\cdot 6^r\|_2=m+3r$ and $\|2^m\cdot 10^r\|_2=m+4r$.
\end{conjecture}

\section{Preparatory Lemmas}
\label{l-complex-section}

Throughout this section, let $l>1$ be a fixed positive integer. The main result that we will prove in this section is that $l$-complexity also grows logarithmically. For convenience, we also adopt the notational convention that $\|1\|_l=0$ for $l>1$. Note that this notation is compatible with that $\|ab\|_l\leq\|a\|_l+\|b\|_l$ when $a$ or $b$ is $1$.

\begin{proposition}
\label{log-growth}
Let $l>1$ and $n\in l\Z^+$. Then, $\log_ln\leq\|n\|_l\leq l\log_ln-1$.
\end{proposition}

The lower bound follows from \Cref{largest-number}, which may seem trivially true but does depend on the assumption that $l>1$; the upper bound follows from \Cref{sum-power}.

\begin{lemma}
\label{largest-number}
Let $l>1$ and $m\in\Z^+$. Then, $l^m$ is the largest number that one can obtain using addition and multiplication with $m$ copies of $l$.
\end{lemma}
\begin{proof}
Indeed, for $l>1$, we always have $a+l\leq a\cdot l$ for all $a\in\Z^+$, so the lemma follows.
\end{proof}

\begin{lemma}
\label{sum-power}
Let $l>1$. Then, 
\begin{align*}
    \|a_1l^{m_1}+a_2l^{m_2}+\cdots+a_kl^{m_k}\|_l\;\leq\; m_1+a_1+a_2+\cdots+a_k-1
\end{align*}
for $m_1>m_2>\cdots>m_k\geq1$ and $0\leq a_i\leq l-1$ with $a_1\neq 0$.
\end{lemma}
\begin{proof}
We prove this by induction on $k$. If $k=1$, then
\begin{align*}
    \|a_1l^{m_1}\|_l\;\leq\;\|l^{m_1-1}\|_l+\|a_1l\|_l\;\leq\;m_1+a_1-1,
\end{align*}
so the result holds. Now, suppose that $k>1$ and that the result holds for $k-1$. Then, 
\begin{align*}
    \|a_1l^{m_1}+\cdots+a_kl^{m_k}\|_l&\;\leq\;\|l^{m_k-1}\|_l+\|a_1l^{m_1-m_k+1}+a_2l^{m_2-m_k+1}+\cdots+a_kl\|_l \\
    &\;\leq\;\|l^{m_k-1}\|_l+\|a_1l^{m_1-m_k+1}+\cdots+a_{k-1}l^{m_{k-1}-m_k+1}\|_l+\|a_kl\|_l \\
    &\;\leq\;(m_k-1)+\big((m_1-m_k+1)+(a_1+\cdots+a_{k-1})-1\big)+a_k \\
    &\eq m_1+a_1+a_2+\cdots+a_k-1,
\end{align*}
where in the third line we use the induction hypothesis for $k-1$. The result then follows from induction.
\end{proof}

\begin{proof}[Proof of \Cref{log-growth}]
The lower bound follows from the fact that $n\leq l^{\|n\|_l}$ by \Cref{largest-number}. For the upper bound, write $n=a_1l^{m_1}+a_2l^{m_2}+\cdots+a_kl^{m_k}$ as in \Cref{sum-power}. Then, $m_1\leq\log_ln$ and $k\leq m_1\leq\log_ln$. Hence, by \Cref{sum-power},
\begin{align*}
    \|n\|_l&\;\leq\;m_1+a_1+a_2+\cdots+a_k-1\\
    &\;\leq\;m_1+(l-1)k-1 \\
    &\;\leq\;\log_ln+(l-1)\log_ln-1\eq l\log_ln-1.
\end{align*}
\end{proof}

\begin{corollary}
\label{two-power}
Let $l>1$. Then, $\|l^m+l^{m^\prime}\|_l=m+1$ for $m\geq m^\prime\geq1$.
\end{corollary}
\begin{proof}
It follows from \Cref{sum-power} that $\|l^m+l^{m^\prime}\|_l\leq m+1$. On the other hand, it follows from \Cref{log-growth} that $\|l^m+l^{m^\prime}\|_l\geq\log_l(l^m+l^{m^\prime})>m$. The result then follows.
\end{proof}

Our last proposition also demonstrates the difference between $l=1$ and $l>1$. This proposition is essentially saying that given $n=b+al$ for $b\in l^2\Z^+$ and $0\leq a\leq l-1$, the most efficient way to write $n$ using $l$ is to first write $b$ using $l$ in the most efficient way, and then write $al$ as $l+\cdots+l$ with $a$ $l$'s.

\begin{proposition}
Let $l>1$. For $n\in l\Z^+$, write $n=b+al$ uniquely for $b\in l^2\Z^+$ and $0\leq a\leq l-1$. Then, $\|n\|_l=\|b\|_l+a$. 
\end{proposition}
\begin{proof}
We will prove this by induction on $n$. If $b=0$, then $\|al\|_l=a$ simply because we can only use addition in this case as $al< l\cdot l$ (or use the recursive definition inductively). Hence, suppose that $b>0$, i.e., $n\geq l^2$. If $a=0$, then indeed this holds, so suppose that $a>0$. Now, suppose that the proposition holds for all $n^\prime<n$. By the recursive definition, there exist $x,y\in l\Z^+$ with $x+y=b+al$ or $xy=b+al$ such that $\|x\|_l+\|y\|_l=\|b+al\|_l$. Since $l^2\nmid b+al$, this can only happen when $x+y=b+al$.

Let $x=b_x+a_xl$ and $y=b_y+a_yl$ with $b_x,b_y\in l^2\Z^+$ and $0\leq a_x,a_y\leq l-1$. By the induction hypothesis for $x,y<n$, 
\begin{align*}
    \|x\|_l\eq\|b_x\|_l+a_x\quad \text{ and }\quad \|y\|_l\eq\|b_y\|_l+a_y.
\end{align*}
Note that we must have $a_x+a_y=a$ or $l+a$. Suppose by contradiction that $a_x+a_y=l+a$. On one hand,
\begin{align*}
    l+1\;\leq\;l+a\eq a_x+a_y\;\leq\;2l-2
\end{align*}
so $l\geq3$. On the other hand,
\begin{align*}
    b+al\eq x+y\eq b_x+b_y+l^2+al
\end{align*}
so
\begin{align*}
    \|b_x\|_l+\|b_y\|_l+a_x+a_y\eq\|x\|_l+\|y\|_l\eq
    \|b+al\|_l&\eq \|b_x+b_y+l^2+al\|_l \\
    &\;\leq\;\|b_x\|_l+\|b_y\|_l+2+a.
\end{align*}
This implies that $l+a=a_x+a_y\leq 2+a$, so $l\leq 2$, which leads to contradiction. Hence, $a_x+a_y=a$ and $b_x+b_y=b$, so
\begin{align*}
    \|b_x\|_l+\|b_y\|_l+a\eq \|x\|_l+\|y\|_l\eq\|b+al\|_l\;\leq\;\|b\|_l+a\;\leq\;\|b_x\|_l+\|b_y\|_l+a.
\end{align*}
The result then follows.
\end{proof}

\section{$2$-complexity}

In this section, we will specialize everything to the case when $l=2$. This section is dedicated to proving \Cref{2-classify} and \Cref{6r-10r}. We first summarize the results that follow from \Cref{l-complex-section}.

\begin{lemma}
\label{2-sum-power}
The following are true:
\begin{enumerate}[label=\rm{(\arabic*)}]
    \item $\|2^{m_1}\|_2=m_1$ and $\|2^{m_1}+2^{m_2}\|_2=m_1+1$ for $m_1>m_2\geq1$;
    \item $\|2^{m_1}+2^{m_2}+\cdots+2^{m_k}\|_2\leq m_1+k-1$ for $m_1>m_2>\cdots>m_k\geq1$;
    \item $\log_2n\leq\|n\|_2\leq2\log_2n-1$.
\end{enumerate}
\end{lemma}

Note that the first part of \Cref{2-sum-power} essentially says that the inequality in the second part is an equality when $k=1,2$. In fact, the equality also holds when $k=3$. To show this, we will first prove the following proposition, which is also the first part of \Cref{2-classify}.

\begin{proposition}
\label{2-two-power-classify}
Let $m\geq 1$ and $n\in 2\Z^+$ with $2^m<n\leq2^{m+1}$. Then, $\|n\|_2=m+1$ if and only if $n$ is of the form $2^m+2^t$ for some $1\leq t\leq m$.
\end{proposition}
\begin{proof}
Indeed, for $1\leq t\leq m$, $\|2^m+2^t\|_2=m+1$ by \Cref{two-power}. Also, the result is trivially true when $m=1$ and easy to verify when $m=2$. We will now prove this by induction on $m$.

Let $m>2$ and suppose that the result holds for all $m^\prime<m$. Let $n\in 2\Z^+$ with $\|n\|_2=m+1$ and $2^m<n\leq2^{m+1}$. Then, there exist $a,b\in 2\Z^+$ such that $n=a+b$ or $n=ab$ with $\|a\|_2+\|b\|_2=\|n\|_2=m+1$. Let $m_a=\|a\|_2$ and $m_b=\|b\|_2$. If $n=a+b$, then $2^m<n\leq 2^{m_a}+2^{m_b}\leq 2^m+2$, so $n=2^m+2$ and hence is of the required form.

Now, suppose $n=ab$. Note that if $a\leq 2^{m_a-1}$, then $n=ab\leq 2^{m_a-1}\cdot2^{m_b}\leq 2^m$, contradicting the assumption. Hence, $2^{m_a-1}<a\leq 2^{m_a}$ and similarly $2^{m_b-1}<b\leq 2^{m_b}$. If $m_a=1$, then $a=2, n=2b$, and $\|b\|_2=m$ with $2^{m-1}<b\leq 2^m$. As $m>m-1\geq 1$, $b=2^{m-1}+2^t$ for some $1\leq t\leq m-1$ by the induction hypothesis for $m-1$, so $n=2b$ is of the required form.

Hence, suppose that $m_a,m_b>1$. For convenience, let $c\in\{a,b\}$. As $m>m_c-1\geq1$ and $2^{m_c-1}<c\leq 2^{m_c}$, we have $c=2^{m_c-1}+2^{t_c}$ for some $1\leq t_c\leq m_c-1$ by the induction hypothesis for $m_c-1$. Then, 
\begin{align*}
    2^{m_a+m_b-1}\;\leq\;n\eq ab&\eq (2^{m_a-1}+2^{t_a})(2^{m_b-1}+2^{t_b}) \\
    &\eq 2^{m_a+m_b-2}+2^{m_a+t_b-1}+2^{m_b+t_a-1}+2^{t_a+t_b}.
\end{align*}
It is easy to see that one must have $t_c\geq m_c-2$ for some $c\in\{a,b\}$ in order for the right hand side to be $\geq 2^{m_a+m_b-1}$. Without loss of generality, suppose that $t_a\geq m_a-2$. If $t_a=m_a-1$, then $n=2^{m_a}\cdot b$ so it is of the required form. Hence, suppose that $t_a=m_a-2$. Then, 
\begin{align*}
   2^{m_a+m_b-1}\;\leq\; n\eq 2^{m_a+m_b-2}+2^{m_a+m_b-3}+2^{m_a+t_b-1}+2^{m_a+t_b-2},
\end{align*}
implying that $t_b\geq m_b-2$. Now, if $t_b=m_b-1$, then $n=2^{m_a+m_b}=2^m+2^m$ is of the required form; if $t_b=m_b-2$, then
\begin{align*}
    n\eq (2^{m_a-1}+2^{m_a-2})(2^{m_b-1}+2^{m_b-2})\eq 2^{m_a+m_b-1}+2^{m_a+m_b-4}
\end{align*}
so $n$ is also of the required form (note that $m_a,m_b>2$ in this case). The result then follows.
\end{proof}

An immediate corollary of this lemma is a slightly better lower bound  for $\|n\|_2$ when $n$ is not expressible as a sum of two powers of $2$, as stated below.

\begin{corollary}
\label{lowerbound}
Let $n\in2\Z^+$. If $n$ is not expressible as a sum of (one or) two powers of $2$, then
\begin{align*}
    \log_2n+1\;\leq\;\|n\|_2.
\end{align*}
\end{corollary}
\begin{proof}
Let $m=\lceil\log_2n\rceil-1$ so that $2^m< n\leq 2^{m+1}$. Note that $\|n\|_2\geq\lceil\log_2n\rceil=m+1$ by \Cref{two-power}, and that $\|n\|_2\neq m+1$ by the assumption and \Cref{2-two-power-classify}, so it follows that \mbox{$\|n\|_2\geq m+2=\lceil\log_2n\rceil+1$}.
\end{proof}

\begin{corollary}
\label{m+2}
Let $m\geq 3$ and $n\in 2\Z^+$ such that $2^m<n\leq 2^{m+1}$. Then, \mbox{$\|n\|_2=m+2$} if $n$ is of one of the following forms:
\begin{enumerate}[label=\rm{(\alph*)}]
    \item $2^{m_1}+2^{m_2}+2^{m_3}$ for $m=m_1>m_2>m_3\geq 1$;
    \item $2^{m_1}+2^{m_2}+2^{m_3}+2^{m_4}$ for $m=m_1>m_2>m_3>m_4\geq2$ with $m_1+m_4=m_2+m_3$;
    \item $2^{m_1}+2^{m_1-3}+2^{m_2}+2^{m_2-1}$ for $m=m_1\geq m_2+3\geq6$.
    \item $2^m+2^{m-5}+2^{m-6}+2^{m-7}$ for $m\geq 10$.
\end{enumerate}
\end{corollary}
\begin{proof}
First, note that in each case $n$ is not expressible as a sum of two powers of~$2$, so $m+2=\lceil\log_2n\rceil+1\leq\|n\|_2$ by \Cref{lowerbound}. It then suffices to show that $\|n\|_2\leq m+2$ in each case, and we will show this case by case.
\begin{enumerate}[label=\rm{(\alph*):}]
    \item It follows from \Cref{2-sum-power} that $\|2^{m_1}+2^{m_2}+2^{m_3}\|_2\leq m_1+3-1=m+2$.
    \item As $m_1+m_4=m_2+m_3$, we have
    \begin{align*}
        2^{m_1}+2^{m_2}+2^{m_3}+2^{m_4}\eq 2^{m_4-2}\cdot(2^{m_2-m_4+1}+2)\cdot(2^{m_3-m_4+1}+2).
    \end{align*}
    It then follows from \Cref{2-sum-power} that 
    \begin{align*}
        \|2^{m_1}+2^{m_2}+2^{m_3}+2^{m_4}\|_2&\;\leq\; m_4-2+(m_2-m_4+2)+(m_3-m_4+2)\\
        &\eq m_1+2\eq m+2.
    \end{align*}
    \item Note that
    \begin{align*}
        2^{m_1}+2^{m_1-3}+2^{m_2}+2^{m_2-1}\eq (2^{m_2-1}+2^{m_2-2})\cdot(2^{m_1-m_2}+2^{m_1-m_2-1}+2).
    \end{align*}
    It then follows from \Cref{2-sum-power} that 
    \begin{align*}
        \|2^{m_1}+2^{m_1-3}+2^{m_2}+2^{m_2-1}\|_2&\;\leq\; m_2+(m_1-m_2+2)\eq m+2.
    \end{align*}
    \item Note that
    \begin{align*}
        2^m+2^{m-5}+2^{m-6}+2^{m-7}\eq (2^{m-6}+2^{m-9})\cdot(2^5+2^4+2^3+2^2).
    \end{align*}
    It then follows from \Cref{2-sum-power} and (b) that
    \begin{align*}
        \|2^m+2^{m-5}+2^{m-6}+2^{m-7}\|_2\;\leq\;(m-5)+7\eq m+2.
    \end{align*}
\end{enumerate}
In conclusion, we always have $\|n\|_2\leq m+2$ in each case. The result then follows.
\end{proof}

As promised before, it follows from this lemma that the inequality in the second part of \Cref{2-sum-power} is an equality when $k=3$ but is no longer an equality when $k\geq 4$. In fact, the four cases described in \Cref{m+2} are the only cases with $2^m<n\leq 2^{m+1}$ and $\|n\|_2=m+2$, as shown in the following proposition, which is also the second part of \Cref{2-classify}.

\begin{proposition}
\label{exactm+2}
Let $m\geq 3$ and $n\in 2\Z^+$ with $2^m<n\leq2^{m+1}$. Then, $\|n\|_2=m+2$ if and only if $n$ is of one of the forms described in \Cref{m+2}.
\end{proposition}
\begin{proof}
Throughout the proof, we also adopt the notation $c\in\{a,b\}$ for convenience as before. We will prove this by induction on $m$. The result is trivially true when $m=3,4$ so suppose that $m>4$ and that the result holds for all $m^\prime<m$.

First, let $a,b\in 2\Z^+$ such that $n=a+b$ with $\|a\|_2+\|b\|_2=\|n\|_2=m+2$. If~$a,b>2$, then $n=a+b\leq 2^{\|a\|_2}+2^{\|b\|_2}\leq 2^m+2^2$, implying that $n=2^m+2$ or $2^m+2^2$, both of which contradict that $\|n\|_2=m+2$ by \Cref{2-sum-power}. Thus, one of $a,b$ must be 2. Without loss of generality, let $a=2$. Then, $\|b\|_2=m+1$ and $2^m-2<b\leq 2^{m+1}-2$. This implies that $2^m<b<2^{m+1}$ so $b=2^m+2^t$ for some $m>t\geq1$ by \Cref{2-two-power-classify}. We have $t>1$ as otherwise $n=2^m+2^2$ contradicting that $\|n\|_2=m+2$. Thus, $n=2^m+2^t+2$ is of the form (a).

Now, let $a,b\in 2\Z^+$ such that $n=ab$ with $\|a\|_2+\|b\|_2=\|n\|_2=m+2$. Let $m_c=\|c\|_2$ for $c\in\{a,b\}$. If $a=2^{m_a}$, then $2^{m-m_a}<b\leq 2^{m-m_a+1}$ and $\|b\|_2=m-m_a+2$, so $b$ is of one of the forms by the induction hypothesis for $m-m_a$, and hence so is $n=2^{m_a}\cdot b$. Thus, suppose that $c<2^{m_c}$ for both $c\in\{a,b\}$, and note that this implies that $m_c>2$ for both $c\in\{a,b\}$. If $a\leq2^{m_a-2}$, then $n=ab\leq 2^{m_a-2}\cdot 2^{m_b}=2^m$, contradicting that $n>2^m$. Hence, $c>2^{m_c-2}$ for both $c\in\{a,b\}$.

To summarize, we are now in the situation where $m_c>2$ and $2^{m_c-2}<c<2^{m_c}$ for both $c\in\{a,b\}$. If $c<2^{m_c-1}$ for both $c\in\{a,b\}$, then $n=ab<2^m$, which leads to contradiction. Suppose for now that $c>2^{m_c-1}$ for both $c\in\{a,b\}$. Then, $2^{m_c-1}<c<2^{m_c}$ and $\|c\|_2=m_c$, so $c=2^{m_c-1}+2^{t_c}$ for some $1\leq t_c<m_c-1$ by \Cref{2-two-power-classify}. Hence,
\begin{align*}
    n\eq ab\eq 2^{m_a+m_b-2}+2^{m_a+t_b-1}+2^{m_b+t_a-1}+2^{t_a+t_b}
\end{align*}
with $m_a+m_b-2>m_a+t_b-1,m_b+t_a-1>t_a+t_b$, so $n$ is of either the form (a) or (b). Thus, suppose without loss of generality that $a>2^{m_a-1}$ and $b<2^{m_b-1}$. As before, we still have $a=2^{m_a-1}+2^{t_a}$ for some $1\leq t_a<m_a-1$.

To summarize again, we are now in the situation where $m_a,m_b>2$, $a=2^{m_a-1}+2^{t_a}$ with $1\leq t_a<m_a-1$, and $2^{m_b-2}<b<2^{m_b-1}$ with $\|b\|_2=m_b$. Note that the condition on~$b$ readily implies that $m_b\geq 5$. By the induction hypothesis for $m_b-2$, $b$ is of one of the forms, so $b\leq 2^{m_b-2}+2^{m_b-3}+2^{m_b-4}+2^{m_b-5}=15\cdot2^{m_b-5}$. Hence, 
\begin{align*}
    2^{m_a-1}+2^{t_a}\eq\frac{n}{b}\;>\;\frac{2^m}{b}\;\geq\;\frac{2^m}{15\cdot 2^{m_b-5}}\eq\frac{2^{m_a+3}}{15}
\end{align*}
so $2^{t_a}>\frac{2^{m_a-1}}{15}$, implying that $t_a\geq m_a-4$.

\textbf{Case i.} $t_a=m_a-4$. Then, 
\begin{align*}
    b\;>\;\frac{2^m}{2^{m_a-1}+2^{m_a-4}}\eq\frac{2^{m_b+2}}{9}\eq\frac{7\cdot2^{m_b+2}}{63}&\;>\;7\cdot 2^{m_b-4}
    \\
    &\eq 2^{m_b-2}+2^{m_b-3}+2^{m_b-4}.
\end{align*}
Since $b$ is of one of the forms, we must have $b=2^{m_b-2}+2^{m_b-3}+2^{m_b-4}+2^{m_b-5}$, so $n$ is of the form (d).

\textbf{Case ii.} $t_a=m_a-3$. 
Then, 
\begin{align*}
    b\thin>\thin\frac{2^m}{2^{m_a-1}+2^{m_a-3}}\thin=\thin\frac{2^{m_b+1}}{5}\thin=\thin\frac{51\cdot2^{m_b+1}}{255}&\thin>\thin51\cdot 2^{m_b-7}\\
    &\thin=\thin2^{m_b-2}+2^{m_b-3}+2^{m_b-6}+2^{m_b-7}.
\end{align*}
Since $b$ is of one of the forms, we must have 
\begin{enumerate}[label=(\arabic*),noitemsep]
    \item $b=2^{m_b-2}+2^{m_b-3}+2^{m_b-5}$ so $n=2^m+2^{m-6}$, contradicting that $\|n\|_2=m+1$;
    \item $b=2^{m_b-2}+2^{m_b-3}+2^{m_b-5}+2^{m_b-6}$ so $n$ is of the form (d);
    \item $b=2^{m_b-2}+2^{m_b-3}+2^{m_b-4}$ so $n$ is of the form (a);
    \item $b=2^{m_b-2}+2^{m_b-3}+2^{m_b-4}+2^{m_b-5}$ so $n$ is also of the form (c).
\end{enumerate}

\textbf{Case iii.} $t_a=m_a-2$.
Then, 
\begin{align*}
    b\;>\;\frac{2^m}{2^{m_a-1}+2^{m_a-2}}\eq\frac{2^{m_b}}{3}\eq\frac{85\cdot2^{m_b}}{255}&\;>\;85\cdot2^{m_b-8}\\
    &\eq 2^{m_b-2}+2^{m_b-4}+2^{m_b-6}+2^{m_b-8}.
\end{align*}
Since $b$ is of one of the forms, we must have
\begin{enumerate}[label=(\arabic*),noitemsep]
    \item $b=2^{m_b-2}+2^{m_b-4}+2^{m_b-5}$ so $n=2^m+2^{m-5}$, contradicting that $\|n\|_2=m+2$;
    \item $b=2^{m_b-2}+2^{m_b-4}+2^{m_b-5}+2^{m_b-7}$ so $n$ is of the form (d);
    \item $b=2^{m_b-2}+2^{m_b-3}+2^{t_b}$ with $1\leq t_b<m_b-3$ so $n$ is of the form (c);
    \item $b=2^{m_b-2}+2^{m_b-3}+2^{t_b}+2^{t_b-1}$ with $t_b<m_b-3$ so $n$ is of the form (b).
\end{enumerate}

In conclusion, $n$ must be of one of the forms described in \Cref{m+2}.
\end{proof}

As before, this proposition also yields a slightly better lower bound for us.

\begin{corollary}
\label{llowerbound}
Let $n\in 2\Z^+$. If $n$ is not expressible as a sum of two powers of $2$ or one of the forms described in \Cref{m+2}, then $\log_2n+2\leq\|n\|_2$. In particular, if $n$ is expressible as a sum of at least five distinct powers of $2$, then $\log_2n+2\leq\|n\|_2$.
\end{corollary}

Having this corollary, we are now ready to prove \Cref{6r-10r}.

\begin{proposition}
\label{6r}
For $m\geq0$, we have
\begin{align*}
    \|2^m\cdot 6^r\|_2\eq m+3r\quad\text{ for }1\leq r\leq 7
\end{align*}
and
\begin{align*}
    m+3r-1\;\leq\;\|2^m\cdot 6^r\|_2\;\leq\; m+3r \quad\text{ for }8\leq r\leq 9.
\end{align*}
\end{proposition}
\begin{proof}
By \Cref{2-sum-power}, if $r=1$, then $\|2^m\cdot 6\|_2=\|2^{m+2}+2^{m+1}\|_2=m+3$, and if $r=2$, then $\|2^m\cdot 6^2\|_2=\|2^{m+5}+2^{m+2}\|_2=m+6$. Now, assume that $3\leq r \leq 9$ so that $n=2^m\cdot(2^2+2)^r$ is not a sum of two powers of $2$. We then have $\log_2n+1\leq\|n\|_2$ by \Cref{lowerbound}, so $m+(\log_26)r+1\leq\|n\|_2$. On the other hand, $\|n\|_2\leq m+r\|6\|_2=m+3r$. Thus, $\lceil m+(\log_26)r+1\rceil\leq\|n\|_2\leq m+3r$. When $r=3,4$, computation shows that $\lceil (\log_26)r+1\rceil=3r$, so the result follows.

Now, assume that $5\leq r\leq 9$. Computation shows that $6^r$ is a sum of at least five distinct powers of $2$ in this case, and so is $n=2^m\cdot 6^r$. We then have $\log_2n+2\leq\|n\|_2$ by \Cref{llowerbound}, so $m+(\log_26)r+2\leq\|n\|_2$. Thus, $\lceil m+(\log_26)r+2\rceil\leq\|n\|_2\leq m+3r$ for $5\leq r\leq 9$. Now, computation shows that $\lceil (\log_26)r+2\rceil=3r$ when $5\leq r\leq 7$, and that $\lceil (\log_26)r+2\rceil=3r-1$ when $8\leq r\leq 9$, so the result follows.
\end{proof}

\begin{proposition}
For $m\geq0$, we have
\begin{align*}
    \|2^m\cdot 10^r\|_2\eq m+4r \quad\text{ for }1\leq r\leq 4
\end{align*}
and for $r=5$,
\begin{align*}
    m+19\;\leq\;\|2^m\cdot 10^5\|_2\;\leq\;m+20.
\end{align*}
\end{proposition}
\begin{proof}
If $r=1$, then $\|2^m\cdot 10\|_2=\|2^{m+3}+2^{m+1}\|_2=m+4$ by \Cref{2-sum-power}; if $r=2$, then $\|2^m\cdot 10^2\|_2=\|2^{m+6}+2^{m+5}+2^{m+2}\|_2=m+8$ by \Cref{m+2}. Now, assume that $3\leq r\leq5$. Computation shows that $10^r$ is a sum of at least five distinct powers of~$2$ in this case, and so is $n=2^m\cdot 10^r$. We then have $\log_2n+2\leq\|n\|_2$ by \Cref{llowerbound}, so $m+(\log_210)r+2\leq\|n\|_2$. On the other hand, $\|n\|_2\leq m+r\|10\|_2=m+4r$. Thus, $\lceil m+(\log_210)r+2\rceil\leq\|n\|_2\leq m+4r$. Now, computation shows that $\lceil (\log_210)r+2\rceil=4r$ when $3\leq r\leq 4$, and that $\lceil (\log_210)r+2\rceil=4r-1$ when $r=5$, so the result follows.
\end{proof}

\section{Further Questions}
\label{questions}

In the final section, we will propose a list of questions on $2$-complexity, and in general $l$-complexity, that seem interesting to investigate further.

\begin{question}
For $l>1$, can one improve the bounds for $\|n\|_l$ in \Cref{log-growth} for all $n\in l\Z^+$ or for a density $1$ subset of $l\Z^+$?
\end{question}

\begin{question}
Can one determine the $l$-complexity of multiples of $l$ that are of certain particular forms? For example, are the following true?
\begin{enumerate}[label=(\arabic*)]
    \item For $l$ not divisible by $2$, $\|(2l)^r\|_l=2r$ for all $r\geq 1$, and if $l>1$, then $\|l^m\cdot (2l)^r\|_l=m+2r$ for all $m\geq 0$ and $r\geq 1$. (Note that $l=1$ is a special case of \Cref{l=1-conj}.)
    \item For $l$ not divisible by $3$, $\|(3l)^r\|_l=3r$ for all $r\geq 1$, and if $l>1$, then $\|l^m\cdot (3l)^r\|_l=m+3r$ for all $m\geq 0$ and $r\geq 1$. (Note that $l=1$ follows from Selfridge's result and $l=2$ is one of the cases of \Cref{l=2-conj}.)
\end{enumerate}
\end{question}

The next question should be compared with the result on integer complexity that for every $n\in\Z^+$, there exists $m_0\geq 0$ such that for all $m\geq m_0$, $\|3^m\cdot n\|_1=3(m-m_0)+\|3^{m_0}\cdot n\|_1$ (see \cite[Theorem~1.5]{altman-zelinsky}).

\begin{question}
Let $l>1$ and define
\begin{align*}
   A_l\deq\big\{n\in l\Z^+\thin\big|\thin\|l^m\cdot n\|_l=m+\|n\|_l\text{ for all }m\geq 0\big\}.
\end{align*}
Can one say anything about the set $A_l$? For example, is it true that for every $n\in l\Z^+$, there exists $m_0\geq 0$ such that $l^{m_0}\cdot n\in A_l$? (Note that e.g. $54\notin A_2$ since $\|54\|_2=8$ and $\|2^2\cdot 54\|_2=\|6^3\|_2=9$, but $2^2\cdot 54\in A_2$ by \Cref{6r-10r}.)
\end{question}

\begin{question}
For $l=2$, should one expect \Cref{l=2-conj} to hold in a more general form, i.e., $\|2^m\cdot(2^u+2)^r\|_2=m+(u+1)r$ for all $m\geq 0$, $u\geq 2$, and $r\geq 1$?
\end{question}

The last question is an interesting interplay between integer complexity and $2$-complexity. Before stating the question, recall that if $\cl{R}$ is a most efficient representation of $n\in\Z^+$ in terms of $1$'s, then $\cl{R}$ does not contain $1+\cdots+1$ with more than five $1$'s. In particular, one can always replace $1+1+1+1$ with $(1+1)(1+1)$ and $1+1+1+1+1$ with $(1+1)(1+1)+1$ to assume that $\cl{R}$ does not contain $1+\cdots+1$ with more than three $1$'s. 

\begin{question}
Let $n\in\Z^+$ and let $\cl{R}$ be a most efficient representation of $n$ in terms of $1$'s that does not contain $1+\cdots+1$ with more than three $1$'s. Let $a(n,\cl{R})$ be the even number obtained by replacing all the $1$'s in $\cl{R}$ with $2$'s. Then, is it true that $\|a(n,\cl{R})\|_2=\|n\|_1$? If not, can one say anything about pairs $(n,\cl{R})$ satisfying or violating this?
\end{question}

\vskip 20pt\noindent {\bf Acknowledgement.} The author would like to thank the anonymous referee(s) for their detailed and useful suggestions.

\end{document}